\documentclass[12pt]{article}

\usepackage[demo]{graphicx} 
\usepackage{hyperref}
\usepackage{float}
\usepackage{indentfirst}
\usepackage{tikz}
\usepackage{amssymb}
\usepackage{amsmath,mathrsfs,amsfonts}
\usepackage{graphics}
\usepackage{amsthm}
\usepackage{setspace}
\usepackage{setspace}
\usepackage[mathscr]{eucal}
\usepackage{caption}
\usepackage{subcaption}
\usepackage{rotating}
\usetikzlibrary{decorations.pathreplacing}
\newtheorem{theo}{Theorem}

\newtheorem{lem}[theo]{Lemma}

\newtheorem{conj}[theo]{Conjecture}

 \theoremstyle{definition}

\theoremstyle{remark}

\newcounter{casenum}[theo]

\newcounter{subcasenum}[theo]

\newcounter{claimnum}[theo]

\allowdisplaybreaks
\begin{document}
\thispagestyle{plain}

\begin{center} {\Large Nonregular graphs with a given maximum degree attaining maximum spectral radius
}
\end{center}
\pagestyle{plain}
\begin{center}
	{
		{\small  Zejun Huang, Jiahui Liu, Chenxi Yang \footnote{  Email: zejunhuang@szu.edu.cn (Huang), mathjiahui@163.com (Liu), 3257812683@qq.com (Yang)}}\\[3mm]
		{\small   School of Mathematical Sciences, Shenzhen University, Shenzhen 518060, China }\\
		
	}
\end{center}
\begin{center}
\begin{minipage}{140mm}
\begin{center}
{\bf Abstract}
\end{center}
{\small    Let $G$ be a connected nonregular graphs of order $n$ with maximum degree $\Delta$ that attains the maximum spectral radius.  Liu and Li (2008) \cite{LL1} proposed a conjecture stating that $G$ has a degree sequence  $(\Delta,\ldots,\Delta,\delta)$ with $\delta<\Delta$. For $\Delta=3$ and $\Delta=4$, Liu (2024) \cite{LLL} confirmed this conjecture by characterizing the structure of  such graphs. Liu also proposed a modified version of the conjecture   for fixed $\Delta$ and sufficiently large $n$, stating that the above $\delta$ satisfies
  \begin{equation*}
  \delta=\left\{
  \begin{aligned}
  &\Delta-1,&&\text{if}~ \Delta \text{ and }   n\ \text{are both  odd};\\
  & 1, &&  \text{if}~     \Delta \text{ is   odd and }   n\text{ is even};\\
  &\Delta-2,&&\text{if}~  \Delta \ \text{is  even}.
  \end{aligned}
  \right.
  \end{equation*}
 For the cases where $\Delta=n-2$ with $n\ge 5$, and $\Delta=n-3$ with $n\ge 59$, we fully characterize the structure of $G$.

}

\end{minipage}
\end{center}
\section{Introduction}
The graphs in this paper are simple, unless otherwise specified. For a graph $G$, we denote its  vertex set by $V(G)$ and its edge set by $E(G)$. The cardinalities of $V(G)$ and $E(G)$ are called the {\it order} and the {\it size} of $G$, respectively. The {\it degree} of a vertex $v$ in a graph  $G$, denoted  $d_G(v)$, is the number of edges incident to $v$ in $G$, with each loop counting as two edges. When the graph $G$ is clear from the context, we  may simply write $d(v)$  to denote $d_G(v)$. We define $\Delta(G)$, $\delta(G)$ and $\bar{d}(G)$ to be the maximum degree, the minimum degree and the average degree of $G$, respectively. The {\it spectral radius} of $G$, denoted $\rho(G)$, is the spectral radius of its adjacency matrix $A(G)$, which is equal to the largest eigenvalue $\lambda_1(G)$ of $G$.

 It is well known that the spectral radius of a graph $G$ is less than or equal to its maximum degree $\Delta(G)$, with equality if and only if $G$ is regular. For  nonregular graphs, an interesting long-standing problem is to determine how far the spectral radius is from the maximum degree. This problem was first investigated by Stevanovi$\rm{\acute{c}}$ \cite{D} in 2004. Let $\lambda_1(n,\Delta)$ be the maximum spectral radius among all connected nonregular graphs of order $n$ with maximum degree $\Delta$.  Stevanovi$\rm{\acute{c}}$  proved
 $$\Delta-\lambda_1(n,\Delta)>\frac{1}{2 n(n \Delta-1) \Delta^2},$$
which was later improved by Zhang \cite{Zhang}, who showed
  $$\Delta-\lambda_1(n,\Delta)>\frac{(\sqrt{\Delta}-\sqrt{\delta})^2}{nD\Delta},$$
 where $D$ is the diameter of $G$. Cioab$\rm{\check{a}}$,  Gregory and  Nikiforov \cite{CGN} proved a stronger lower bound on $\Delta-\lambda_1(n,\Delta)$:
$$\Delta-\lambda_1(n,\Delta)>\frac{n\Delta-2m}{n(D(n\Delta-2m)+1)},$$
where $m$ is the size of $G$.  They also conjectured that $\Delta-\lambda_1(n,\Delta)>1/(nD)$, which was proved by  Cioab$\rm{\check{a}}$ \cite{SMC}. Liu, Shen and Wang \cite{LSW} derived another lower bound:
 $$\Delta-\lambda_1(n,\Delta)>\frac{\Delta+1}{n(3n+2\Delta-4)},$$
 which was slightly improved by Liu, Huang and You \cite{LHY}.
In \cite{LSW}, the authors also obtained an upper bound on $\Delta-\lambda_1(n,\Delta)$, leading  to the result $$\Delta-\lambda_1(n,\Delta)=\Theta\left(\frac{\Delta}{n^2}\right).$$
Moreover, they conjectured that for each fixed $\Delta$,
 $$ \lim_{n \rightarrow \infty}\frac{n^2(\Delta-\lambda_1(n,\Delta))}{\Delta-1}=\pi^2,$$
 which was disproved by Liu \cite{LLL} and modified to be
 $$ \lim_{n \rightarrow \infty}\frac{n^2(\Delta-\lambda_1(n,\Delta))}{\Delta-1}=\frac{\pi^2}{4}\quad \text{for an odd number } \Delta\ge 3$$
 and
$$ \lim_{n \rightarrow \infty}\frac{n^2(\Delta-\lambda_1(n,\Delta))}{\Delta-2}=\frac{\pi^2}{2}\quad \text{for an even number } \Delta\ge 3.$$
For more related  results on this topic, readers may refer to \cite{AS,CH,FZ, XL, WZ} and the references therein.

Another direction in studying the above problem is to investigate the degree sequence or the exact structures of the extremal graphs. Let $\mathcal{G}(n,\Delta)$ denote the set of connected nonregular graphs of order $n$ with maximum degree $\Delta$ that attain the  maximum spectral radius. Liu and Li \cite{LL1} posed the following conjecture:
\begin{conj}\cite{LL1}\label{conj1}
  Suppose $3\le\Delta\le n-2$ and $G\in\mathcal{G}(n,\Delta)$. Then the degree sequence of $G$ is $(\Delta,\dots,\Delta,\delta)$, where
  \begin{equation}
  \nonumber
  \delta=\left\{
  \begin{aligned}
  &\Delta-1,&&if~ n\Delta \ is \ odd; \\
  &\Delta-2,&&if~ n\Delta \ is \ even. \\
  \end{aligned}
  \right.
  \end{equation}
\end{conj}
For $\Delta=3,4$, Liu \cite{LLL} confirmed Conjecture \ref{conj1} and determined the exact structure of the extremal graphs in $\mathcal{G}(n,\Delta)$.   Liu also modified Conjecture \ref{conj1} as follows:
\begin{conj}\cite{LLL}\label{conj2}
  Suppose $\Delta\ge3$ and $G\in\mathcal{G}(n,\Delta)$. Then for each fixed $\Delta$ and sufficiently large $n$, the degree sequence of $G$ is $(\Delta,\dots,\Delta,\delta)$, where
 \begin{equation*}
  \delta=\left\{
  \begin{aligned}
  &\Delta-1,&&\text{if}~ \Delta \text{ and }   n\ \text{are both  odd};\\
  & 1, &&  \text{if}~     \Delta \text{ is   odd and }   n\text{ is even};\\
  &\Delta-2,&&\text{if}~  \Delta \ \text{is  even}.
  \end{aligned}
  \right.
  \end{equation*}
\end{conj}

In this paper, we study the structures of graphs in $\mathcal{G}(n,\Delta)$ when $\Delta$ is close to $n$. To state our main results, we introduce the following notations.

 Two graphs $ G$ and $ H $ are said to be \textit{isomorphic}, denoted $ G \cong H $, if there exists a bijection $\phi : V(G) \to V(H)$ such that for all $u, v \in V(G)$, $uv \in E(G) $ if and only if $\phi(u)\phi(v) \in E(H)$.

Given a vertex set $S\subseteq V(G)$, we denote by $G[S]$ the subgraph of $G$ induced by $S$. A {\it $k$-matching}, denoted $M_k$, is the disjoint union of $k$ edges.
The graph  $K_t-M_{t/2}$ is a graph obtained from the complete graph $K_t$ by deleting the edges of a matching $M_{t/2}$, where $t$ is even.

Let $G(n,t)$ be a graph of order $n$ constructed from the disjoint union of a vertex $u$,  $K_t-M_{t/2}$ and $K_{n-t-1}$, by adding all possible edges between $u$ and $K_t-M_{t/2}$, as well as all possible edges between $K_t-M_{t/2}$ and $K_{n-t-1}$.

Let $H_1(n)$ be a graph of order $n$, whose vertex set  can be partitioned as $V=\{u,v\}\cup V_1 \cup V_2$ such that:
\begin{itemize}
\item $uv\in E(H_1(n))$,
\item \vskip -0.3cm $H_1(n)[V_1]=K_{n-4}-M_{({n-4})/{2}}$,
  \item \vskip -0.3cm $H_1(n)[V_2]=K_2$,
  \item \vskip -0.3cm every vertex in $V_1$ is adjacent to $v$, and
  \item \vskip -0.3cm every vertex in $V_2$ is adjacent to all vertices in $V_1$.
 \end{itemize}
 Let $H_2(n)$ be a graph of order $n$, whose vertex set can be partitioned as $V=\{u,v_1,v_2\}\cup V_1 \cup V_2\cup V_3$, where $|V_2|=|V_3|=2$, and the following conditions hold:
\begin{itemize}
\item  \vskip -0.3cm$H_2(n)[\{u,v_1,v_2\}]=K_3$,
\item \vskip -0.3cm$H_2(n)[V_2\cup V_3]= K_4$,
 \item \vskip -0.3cm$H_2(n)[V_3]=K_{n-7}-M_{({n-7})/{2}}$,
 \item \vskip -0.3cmevery vertex in $V_3$ is adjacent to all vertices in $\{v_1,v_2\}\cup V_1\cup V_2$,
 \item \vskip -0.3cm$v_1$ is adjacent to both vertices in $V_1$, and
  \item \vskip -0.3cm$v_2$ is adjacent to both vertices in $V_2$.
   \end{itemize}
   See Figure \ref{fig:2} for the above graphs $G(n,t), H_1(n)$ and $H_2(n)$.
\begin{figure}[H]
	\centering
	\begin{subfigure}{.45\textwidth}
		\centering
		\begin{tikzpicture}
	    \draw  (3,6) ellipse [x radius=1cm, y radius=0.7cm];
		\draw  (6,6) ellipse [x radius=1cm, y radius=0.7cm];
		\draw(0,6)[fill]circle[radius=1.0mm];
		\node at(3,6){$K_t-M_{\frac{t}{2}}$};
		\node at (6,6){$K_{n-t-1}$};
		\draw(0,6)--(2,6);
		\draw(4,6)--(5,6);
		\node at(0,6)[above=3pt]{$u$};	
	\node at(3,4.5)[above=3pt]{$G(n,t)$};
	
	\draw  (3,3.5) ellipse [x radius=1.5cm, y radius=0.7cm];
		\draw  (6.5,3.5) ellipse [x radius=1cm, y radius=0.7cm];
		\draw(-0.5,3.5)[fill]circle[radius=1.0mm];
        \draw(0.5,3.5)[fill]circle[radius=1.0mm];
		\node at(3,3.5){$K_{n-4}-M_{\frac{n-4}{2}}$};
		\node at (6.5,3.5){$K_2$};
		\node at(-0.5,3.5)[above=3pt]{$u$};
        \node at(0.5,3.5)[above=3pt]{$v$};
		\draw(-0.5,3.5)--(0.5,3.5);
        \draw(0.5,3.5)--(1.5,3.5);
		\draw(4.5,3.5)--(5.5,3.5); 	
		\node at (3,2.5){$H_1(n)$};
	\end{tikzpicture}
\end{subfigure}
\begin{subfigure}{.5\textwidth}
	\centering
	\vspace{-15pt}
	\begin{tikzpicture}
	    \draw(4,0) ellipse [x radius=1.5cm, y radius=0.7cm];
        \draw(4,1.5) ellipse [x radius=0.7cm, y radius=0.5cm];
        \draw(4,-1.5) ellipse [x radius=0.7cm, y radius=0.5cm];
		\draw(1.5,1)[fill]circle[radius=1.0mm];
        \draw(1.5,-1)[fill]circle[radius=1.0mm];
        \draw(0,0)[fill]circle[radius=1.0mm];
        \node at(4,0){$K_{n-7}-M_{\frac{n-7}{2}}$};
        \node at(4,1.5){$K_2$};
        \node at(4,-1.5){$K_2$};
        \node at(0,0)[below=2pt]{$u$};
        \node at(1.5,1)[left=2pt]{$v_1$};
        \node at(1.5,-1)[left=2pt]{$v_2$};
        \draw (0,0)--(1.5,1);
        \draw (0,0)--(1.5,-1);
        \draw (1.5,1)--(1.5,-1);
        \draw (4,1)--(4,0.7);
        \draw (4,-1)--(4,-0.7);
        \draw (1.5,1)--(3.3,1.5);
        \draw (1.5,-1)--(3.3,-1.5);
        \draw (1.5,1)--(2.5,0);
        \draw (1.5,-1)--(2.5,0);
        \draw (4.7,-1.5) arc(-90:90:1.5);
	\end{tikzpicture}
	\caption*{$H_2(n)$}
\end{subfigure}
	\caption{$G(n,t)$, $H_1(n)$, $H_2(n)$}
	\label{fig:2}
\end{figure}
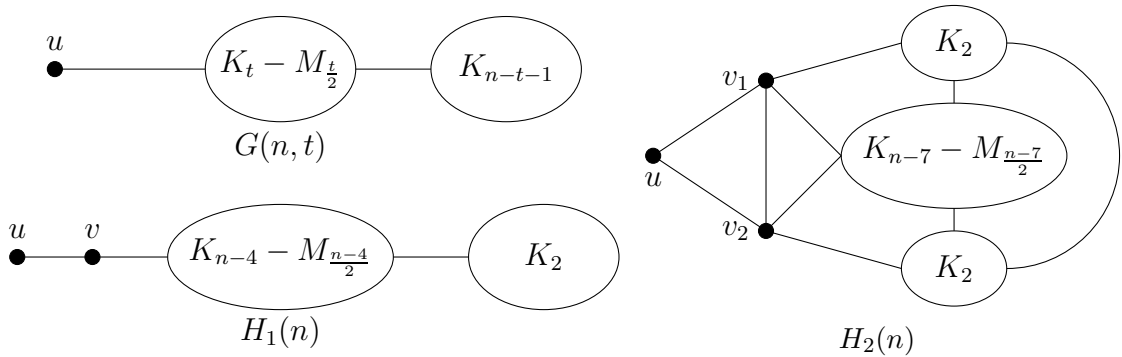

\noindent Our main results states as follows.

\begin{theo}\label{theo7}
    Let $n\ge 5$ be an integer. Then  $G\in \mathcal{G}(n,n-2)$ if and only if
    \begin{equation}
    \nonumber
   G\cong \left\{
    \begin{aligned}
    & G(n,n-3),&&      if~  n\ is\ odd;\\
    &G(n,n-4)~ or\ G(n,2),&& if~  n\ is\ even.
    \end{aligned}
    \right.
    \end{equation}
\end{theo}

\begin{theo}\label{theo8}
	 Let $n\ge 59$ be an integer. Then  $G\in \mathcal{G}(n,n-3)$ if and only if $G\cong H_1(n)$ when $n$ is even and $G\cong H_2(n)$ when $n$ is odd.
\end{theo}

Theorem \ref{theo8} shows that for $\Delta=n-3$, the  connected graph of order $n$ with maximum degree $\Delta$ attaining maximum spectral radius is unique and it has a degree sequence $(\Delta,\ldots,\Delta,\delta)$, where $\delta$ is as small as possible.

We give some preliminaries   in Section 2 and  present the proofs in Section 3 and Section 4.

\section{Preliminary}
In this section, we introduce some notations and preliminary lemmas that will be used throughout the paper.

Given a graph $G=(V,E)$, let $S\subseteq V$.
We denote by $G-S$ the induced subgraph $G[V\setminus S]$, which is also written as $G-v$ when $S=\{v\}$ is a singleton. The {\it neighbourhood} of $S$, denoted  $N(S)$, is the set of vertices adjacent to some vertices in $S$, while the {\it closed neighbourhood} of $S$  is $N[S]:=N(S)\cup S$. We denoted by $\overline{G}$ the {\it complement} of $G$, and by $G^{loop}$, the graph obtained by adding a loop incident to each vertex of $G$. A loop incident to a vertex $v$ is denoted by $loop(v)$.

Given a graph $G=(V,E)$ that allows loops,  its adjacency matrix is defined as
 $A(G)=(a_{ij})$, where
 $$a_{ij}=\left\{\begin{array}{ll}
 2,&v_iv_j\in E \text{ and } i=j ;\\
 1,&v_iv_j\in E \text{ and } i\ne j ;\\
 0,&v_iv_j\not\in E.
 \end{array}
 \right.$$
The {\it characteristic polynomial} of a graph $G$, denoted   $P(G,\lambda)$, is the characteristic polynomial of  $A(G)$.   When $G$ is connected, by the Perron-Frobenius theorem, there is a unique positive unit eigenvector of $A(G)$ corresponding to the eigenvalue $\rho(G)$, which is called the {\it Perron vector} of $G$.

For a graph $G$ with maximum degree $\Delta$, we always denote by $$S=\{v\in V(G):d(v)<\Delta\} \quad \text {and}\quad T=\{v\in V(G):d(v)=\Delta\}.$$
Let $X=(x_1,x_2,\dots,x_n)^T$ be the Perron vector of $G$ with $x_k$   corresponding to the vertex $k\in V(G)$. We have the following nice properties on $S$ and $T$ for graphs in $\mathcal{G}(n,\Delta)$.
\begin{lem}\cite{LL2}\label{lemm2}
Let $G\in \mathcal{G}(n,\Delta)$. Then $G[S]$ is a complete graph.
\end{lem}
\begin{lem}\label{lemm4}\cite{LLL}
Let $G\in\mathcal{G}(n,\Delta)$ and $u,v\in S$. Then $x_v\leq x_u$ if and only if $N_T(v)\subseteq N_T(u)$.
\end{lem}
\begin{lem}\label{lemm5}\cite{LLL}
Let $G\in\mathcal{G}(n,\Delta)$ and $s\in S,t\in T$ be two vertices. Then $x_s<x_t$.
\end{lem}
\begin{lem}\cite{CR}\label{lemm3}
Let $G$ be a connected graph with $V(G)=[n]$. Suppose $s,t,u,v\in V(G)$ are distinct vertices with  $uv,st\in E(G)$ and $sv,tu \notin E(G)$.
The local switching $LS(G;s,t,v,u)$ of $G$ is the operation that replaces the edges $uv$ and $st$ in $G$ with the edges $sv$ and $tu$. The resulting  graph is denoted by $G'=LS(G;s,t,v,u)$.
If $(x_s-x_u)(x_v-x_t)\geq 0$, then $\rho(G')\geq \rho(G)$, with equality if and only if $x_s=x_u$ and $x_v=x_t$.
\end{lem}
Suppose $A$ is a symmetric real matrix whose rows and columns are indexed by $V=\{1, \ldots, n\}$. Let $\left\{V_1, \ldots, V_m\right\}$ be a partition of $V$.  Denote by  $A_{ij}$ the submatrix of $A$ with row indices from $V_i$ and  column indices  from $V_j$. Let $b_{ij}$ represent the average row sum of $A_{ij}$. The matrix $B=\left(b_{ij}\right)_{m\times m}$ is called the {\it quotient matrix} of $A$. If the row sum of each block $A_{ij}$ is constant, then the partition is called {\it equitable}.

\begin{lem}\label{lemm6}
\cite{WH} Let $G$ be a (multi-)graph that allows loops, and let $\pi=\{V_1, \ldots, V_r\}$ be a partition of $V(G)$ with quotient matrix $B=(b_{i, j})$. Then
$$
\rho(G) \geq \rho(B),
$$
with equality if and only if the partition is equitable. If the partition $\pi$ is equitable, then each eigenvalue of $B$ is an eigenvalue of $G$.
\end{lem}

\begin{lem}
Let $A$ and $A^{loop}$ be the adjacent matrices of $G$ and  $G^{loop}$, respectively. Suppose $\pi=\{V_1, \ldots, V_r\}$ is an equitable partition of $V(G)$ with quotient matrix $B$. Then $\pi$ is also an equitable partition of $V(G^{loop})$ with quotient matrix $B^{loop}$ such that  $$\rho(A^{loop})=\rho(B^{loop})=\rho(A)+2=\rho(B)+2 .$$ Moreover, $G$ and $G^{loop}$ has the same
 Perron vector.
\end{lem}
\begin{proof}
It is clear that $\pi$ is  an equitable partition of $V(G^{loop})$. Moreover, we have
  $$A^{loop}=A+2I \quad \text{and}\quad B^{loop}=B+2I.$$
  Therefore, $G$ and $G^{loop}$ have the same Perron vector. By Lemma \ref{lemm6},  we get
  $$\rho(A^{loop})=\rho(A)+2=\rho(B)+2=\rho(B^{loop}).$$
\end{proof}

\begin{lem}(\cite{CR})\label{lemm8}
	Let $G_1,G_2$ be two graphs, If $P(G_2,\lambda)\geq P(G_1,\lambda)$
	for all $\lambda \geq \rho(G_1)$, then $\rho(G_2)\leq \rho(G_1)$, with equality if and only if $P(G_2,\rho(G_1))= P(G_1,\rho(G_1))=0$.
\end{lem}
Similarly as Lemma \ref{lemma8}, we have the following.
\begin{lem}\label{lemma8}
	Let $A_1,A_2$ be  real matrices such that $\rho(A_i)$ equals   the maximum real eigenvalue of $A_i$, and $\rho(A_i)\leq \rho,~i=1,2$. If $P(A_2,\lambda)\geq P(A_1,\lambda)$
	for all $\lambda \in [\rho(A_1),\rho]$, then $\rho(A_2)\leq \rho(A_1)$, with equality if and only if $P(A_2,\rho(A_1))=P(A_1,\rho(A_1))=0$.
\end{lem}
\begin{proof}
  Since $P(A_2,\lambda)\geq P(A_1,\lambda)>0$
	for all $\lambda \in (\rho(A_1),\rho]$, the real roots of  $P(A_2,\lambda)$ are located in  $(-\infty,\rho(A_1)]$, with $\rho(A_1)$ being a root if and only if  $P(A_2,\rho(A_1))=0$. Therefore, the result follows.
\end{proof}
\begin{lem}\label{lemma9}
    Let $A_1,A_2$ be  real matrices such that $\rho(A_i)$ equals the maximum real eigenvalue of $A_i$, and $\rho(A_i)\leq \rho,~i=1,2$. If $P(A_2,\lambda-k)-P(A_1,\lambda)\ge0$ for all $\lambda\in [\rho(A_1),\rho+k]$, then $\rho(A_2)+k\le\rho(A_1)$.
\end{lem}
\begin{proof}
    By the definition of $P(A_2,\lambda)$, we have
    $$P(A_2,\lambda-k)=\det((\lambda-k)I-A_2)=\det(\lambda I-(A_2+kI))=P(A_2+kI,\lambda).$$
Since $$P(A_2+kI,\lambda)-P(A_1,\lambda)\geq 0 ~\text{for~all}~\lambda \in [\rho(A_1),\rho+k],$$
   applying Lemma \ref{lemma8}, we have $\rho(A_2+kI)=\rho(A_2)+k\leq \rho(A_1)$.
\end{proof}

\begin{lem}\label{lemma7}\cite{Zhan}
	Let $A$ be an $n\times n$ symmetric real matrix. Then $$\rho(A)=\max\limits_{x\in\mathbb{R}^n,||x||=1}x^TAx.$$
\end{lem}
\begin{lem}(Perron-Frobenius Theorem, \cite{Zhan})
    If $A$ is an irreducible  nonnegative square matrix, then
    $\rho(A)$ is a simple  eigenvalue of $A$.
\end{lem}
\begin{lem}\label{lemma16}
Let $G$ be a connected graph with maximum degree $\Delta$. Suppose $X$ is the Perron vector of $G$ with maximum component $x$. Then $\rho(G)x<\sqrt{\Delta}$.
\end{lem}
\begin{proof}
Suppose  $X=(x_1,\ldots,x_n)$ with $x_i$ corresponding to the vertex $i\in V(G)$. Let $u$ be a vertex such that $x_u=x$ and  $N_G(u)=\{w_1,w_2,\dots,w_{d}\}$. Let $Y=(x_{w_1},\dots,x_{w_{d}})^T$,
and let $J_{d}$ be the $d\times d$ matrix whose entries are all ones. Since $dI-J_{d}$ is positive semidefinite,
we have $Y^T(dI-J_{n-3})Y\ge 0$.
By the characteristic equation  of $G$, we have
$\rho(G)x=\Sigma_{i=1}^{d}x_{w_i}$, which leads to $$\rho^2(G) {x}^2=(\Sigma_{i=1}^{d}x_{w_i})^2=Y^TJ_{d}Y.$$
 It follows that $$\rho^2(G) {x}^2=Y^TJ_{d}Y\le dY^TY<d X^TX=d\le \Delta,$$ which implies
    \begin{align*}
	\rho(G)x<\sqrt{\Delta}.
	\end{align*}
\end{proof}
\section{Proof of Theorem \ref{theo7}}
\begin{proof}
Suppose  $G\in \mathcal{G}(n,n-2)$ with
$V(G)=[n]$. Let $ X=(x_1,x_2,\dots,x_n)^T$ be the Perron vector of $G$ with $x_k$  corresponding to the vertex $k\in V(G)$.

Since $G$ is nonregular, we have $|S|\geq 1$.
Suppose $|S|\geq 2$. Let $u,v\in S$ be distinct vertices with $x_u=\max_{i\in S} x_i$.
Since $d(w)=\Delta=n-2$ for all $ w\in T$,   $w$ must be adjacent to at least one of  $u,v$.
Since $x_u\geq x_v$, applying Lemma \ref{lemm4}  we have $N_T(u)\supseteq N_T(v)$, which implies  $w\in N(u)$ for all $w\in T$. Therefore, we have $d_T(u)=|T|$. On the other hand,
by Lemma \ref{lemm2} we have $d_S(u)=|S|-1$. Hence,
 $$d(u)=d_T(u)+d_S(u)=|T|+|S|-1=n-1>\Delta=n-2,$$ a contradiction. Therefore, we get
$|S|=1$.

Suppose $S=\{u\}$ with $d(u)=\delta<n-2=\Delta$.
Since $\sum_{v\in G}d(v)=\delta+(n-1)(n-2)$ is even and $G$ is  connected,  $\delta$ is a positive even number.
Let $$T_1=N(u),T_2=V(G)\setminus N[u].$$ Then $\{T_1, T_2\}$ is a partition of $T$ with $|T_1|=\delta\ge 2$ and $|T_2|=n-\delta-1\ge 2$.
Notice that $d(w)=\Delta=n-2$ for all $w\in T$. For every $v \in T_2$,
 $v$ is adjacent to all the vertices in $T\setminus\{v\}$; for evert vertex $v\in T_1=N(u)$,
  $v$ is  adjacent to all but one vertex in $T\setminus\{v\}$.
Therefore, we have $$G[T_1]= K_\delta-M_{\delta/2} \quad {\rm and}\quad G[T_2]=  K_{n-\delta-1}.$$ Moreover, $v_1v_2\in E(G)$ for all
 $v_1\in T_1,v_2\in T_2$. Hence, $G$   is isomorphic to the following graph $G(n,\delta)$; see Figure \ref{fig:1}.

\begin{figure}[H]
	\centering
	\begin{tikzpicture}
	    \draw  (3,1) ellipse [x radius=1cm, y radius=0.8cm];
		\draw  (6,1) ellipse [x radius=1cm, y radius=0.8cm];
		\draw(0,1)[fill]circle[radius=1.0mm];
		\node at(3,1){$K_\delta-M_{\delta/2}$};
		\node at (6,1){$K_{n-\delta-1}$};
		\draw(0,1)--(2,1);
		\draw(4,1)--(5,1);
		\node at(0,0.2)[above=3pt]{$u$};
		\node at(3,-0.8)[above=3pt]{$N(u)$};
		\node at(6,-0.8)[above=3pt]{$V(G)\setminus N[u]$};
	\end{tikzpicture}
	\caption{$G(n,\delta)$}
	\label{fig:1}
\end{figure}
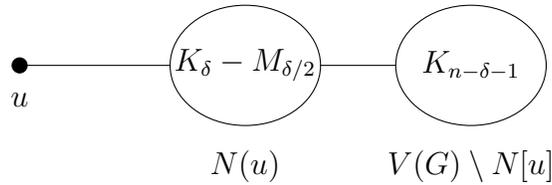

Notice that $\pi=\{\{u\},T_1,T_2\}$ is an  equitable partition of $V(G)$ with
 quotient matrix
\begin{align*}
A_\delta=&
\begin{pmatrix}
0 & \delta & 0 \\
1 & \delta-2 & n-\delta-1 \\
0 & \delta & n-\delta-2
\end{pmatrix},
\end{align*}
whose characteristic polynomial is
\begin{eqnarray*}
det(\lambda I-A_\delta)&=&
\begin{vmatrix}
\lambda & -\delta & 0\\
-1 & \lambda-\delta+2 & -n+\delta+1\\
0 & -\delta & \lambda-n+\delta+2
\end{vmatrix} =\begin{vmatrix}
\lambda & -\delta & 0\\
-1 & \lambda+2 & -\lambda-1\\
0 & -\delta & \lambda-n+\delta+2
\end{vmatrix}\\
&=&\lambda^3+(4-n)\lambda^2+(4-2n)\lambda-(\delta^2+2\delta-n\delta)\equiv f(\delta,\lambda).
\end{eqnarray*}
Applying Lemma \ref{lemm6}  we have $\rho(G(n,\delta))=\rho(A_\delta)$. Since $A_\delta$ is an irreducible nonnegative matrix,
by the Perron-Frobenius theorem, $\rho(A_\delta)$ equals the maximum real eigenvalue of $A_{\delta}$.

Given two integers $\delta_1,\delta_2$, we have
$$f(\delta_2,\lambda)-f(\delta_1,\lambda)=\delta_1^2+2\delta_1-n\delta_1-\delta_2^2-2\delta_2+n\delta_2=(\delta_1-\delta_2)(\delta_1+\delta_2-n+2).\label{equ11}$$
Recall that $\delta$ is even and $2\leq \delta \leq n-3$.
If $n$ is odd, then $$f(\delta,\lambda)-f(n-3,\lambda)\geq 0 \quad \text{  for all real number}~  \lambda,$$ with equality if and only if $\delta=n-3$.
Applying Lemma \ref{lemma8}, we have $$\rho(G(n,\delta))=\rho(A_\delta)\leq \rho(A_{n-3})=\rho(G(n,{n-3})),$$ with equality if and only if $\delta=n-3$. Therefore, $G\in\mathcal{G}(n,n-2)$ if and only if $G\cong G(n,n-3)$.

If $n$ is even, then $$f(\delta,\lambda)-f(n-4,\lambda)\geq 0 \quad \text{  for all real number}~  \lambda,$$ with equality if and only if $\delta\in\{2,n-4\}$.
By Lemma \ref{lemma8}, we have $$\rho(G(n,\delta))=\rho(A_\delta)\leq \rho(A_{n-4})=\rho(G(n,{n-4})),$$ with equality if and only if $\delta\in\{2,n-4\}$.
Therefore, $G\in\mathcal{G}(n,n-2)$ if and only if $G\cong G(n,2)$ or $G\cong G(n,n-4)$. This completes the proof.
\end{proof}
\section{Proof of Theorem \ref{theo8}}
\begin{proof}
 Suppose $G\in \mathcal{G}(n,n-3)$.   We assert that $|S|\in \{1,2\}$. Let $ \alpha=(\alpha_1,\alpha_2,\dots,\alpha_n)^T$ be the Perron vector of $G$ with $x_k$  corresponding to the vertex $k\in V(G)$.
 Since $G$ is nonregular, we have $|S|\geq 1$.
 Suppose  $S$ has three distinct vertices $u_1,u_2,u_3$ with $\alpha_{u_1}=\max\{\alpha_i|i\in S\}$.
 Since $d(w)=\Delta=n-3$ for all $w\in T$,  $w$ must be adjacent to at least one of the vertices $u_1,u_2,u_3$.
 By Lemma \ref{lemm4}, we can deduce $w\in N_T(u_1)$  for all $w\in T$, which implies $d_T(u_1)=|T|$.
 On the other hand, Lemma \ref{lemm2} implies $d_S(u_1)=|S|-1$. Hence, we have
 $$d(u_1)=d_T(u_1)+d_S(u_1)=|T|+|S|-1=n-1,$$ which contradicts $\Delta=n-3$. Therefore,   $|S|\in \{1,2\}$. We distinguish two cases.

{\it Case 1.} $|S|=1$, say, $S=\{u\}$.
Then $G\in \mathcal{G}(n,n-3)$ has degree sequence ($\Delta,\dots,\Delta,\delta$) with $\delta<\Delta=n-3$.
Since $\sum_{v\in V(G)}d(v)=(n-1)(n-3)+\delta$,   $\delta$ is odd when $n$ is even and $\delta$ is even when $n$ is odd, which implies $\delta\le n-5$. Now we deduce the possible   spectral radius of  $G$ for different choices of $\delta$.

{\it Subcase 1.1.} $\delta=1$. Then $n$ is even.
Suppose $u\in V(G)$ with   $N(u)=\{v_1\}$. Since $d(v_1)=n-3$, there are exactly two vertices $v_2,v_3$ not in the closed neighborhood of $v_1$, i.e., $ N[v_1]=V(G)\setminus\{v_2,v_3\}$. Let $T'\equiv V(G)\setminus\{u,v_1,v_2,v_3\}$. Then $N(v_1)=\{u\}\cup T'$. Since
 $d_{G}(v)=\Delta=n-3$ for all $v\in V(G)\setminus \{u\}$, $v_2$ and $ v_3$ are  adjacent to all the  vertices in $V(G)\setminus \{u,v_1\}$, which implies that $T'$ induces a graph obtained from $K_{n-4}$ by deleting a perfect matching, i.e., $G[T']=K_{n-4}-M_{(n-4)/2}$.
Therefore, $G$ is isomorphic to the following graph $G_1$, which is isomorphic to $H_1(n)$;  See Figure \ref{fig:3}.
\begin{figure}[H]
	\centering
	\begin{subfigure}{.45\textwidth}
	\centering
	\begin{tikzpicture}
	\draw  (2.5,1) ellipse [x radius=1.0cm, y radius=0.7cm];
	\draw(-0.5,1)[fill]circle[radius=1.0mm];
	\draw(0.5,1)[fill]circle[radius=1.0mm];
    \draw(4.5,1.5)[fill]circle[radius=1.0mm];
    \draw(4.5,0.5)[fill]circle[radius=1.0mm];
	\node at(2.5,1){$T'$};
	\draw(-0.5,1)--(0.5,1);
	\draw(0.5,1)--(1.5,1);
	\draw(3.4,1.3)--(4.5,1.5);
    \draw(3.4,0.7)--(4.5,0.5);
	\draw(4.5,1.5)--(4.5,0.5);
	\node at(-0.5,1)[above=3pt]{$u$};
	\node at(0.5,1)[above=3pt]{$v_1$};
	\node at(4.5,1.5)[right]{$v_2$};
	\node at(4.5,0.5)[right]{$v_3$};
	\end{tikzpicture}
	\end{subfigure}
\caption{$G_1$, where $ {G[T']}=K_{n-4}-M_{\frac{n-4}{2}}$}
	\label{fig:3}
\end{figure}
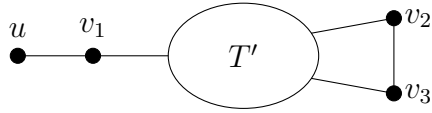

Since $n\ge 59$, we have $T'\neq\emptyset$. Therefore, $\pi=\{\{u\},\{v_1\},\{v_2,v_3\},T'\}$ is an equitable partition of $V(G_1)$ with quotient matrix
\begin{align*}
B_1&=
\begin{pmatrix}
0 & 1 & 0 & 0\\
1 & 0 & 0 & n-4 \\
0 & 0 & 1 & n-4\\
0 & 1 & 2 & n-6
\end{pmatrix},
\end{align*}
 whose characteristic polynomial   is
 \begin{equation}\label{Eq01}
 P(B_1,\lambda)=\lambda^4+(5-n)\lambda^3+(5-2n)\lambda^2+(2n-9)\lambda+n-2\equiv f_1(\lambda).
 \end{equation}
Moreover, $\rho(G)=\rho(G_1)= \rho(B_1)$ is the maximum real eigenvalue of $B_1$, as $B_1$ is an irreducible  nonnegative matrix.

{\it Subcase 1.2.} $\delta=2$.  Then $n$ is odd. Suppose $u\in V(G)$ with $N(u)=\{v_1,v_2\}$. We distinguish two subcases.

{\it Subcase 1.2.1.} $v_1v_2\not\in E(G)$. Recall that $d(v_1)=d(v_2)=n-3$. Suppose
 $V(G)-N[v_1]= \{v_2,v_3\},  V(G)-N[v_2]= \{v_1,v_4\}.$ Then $$ N[v_3]=V(G)\setminus \{u,v_1\}, \quad N[v_4]=V(G)\setminus \{u,v_2\}.$$
Let $T'=V(G)\setminus\{u,v_i, i=1,2,3,4\}$.  Since $u\not\in N(v)$ and $ \{v_1,v_2,v_3,v_4\}\subseteq N(v)$ for all $v\in T'$, we  have $G[T']=K_{n-5}-M_{(n-5)/2}$. Therefore, $G\equiv G_2^{(1)}$ has the following diagram; see
Figure \ref{fig:4}.
\begin{figure}[H]
	\centering
	\begin{subfigure}{.4\textwidth}
	\centering	
	\begin{tikzpicture}
	\draw(2.5,0) ellipse [x radius=1.5cm, y radius=1cm];
	\draw(0,1.5)[fill]circle[radius=1.0mm];
	\draw(0,0.5)[fill]circle[radius=1.0mm];
	\draw(0,-0.5)[fill]circle[radius=1.0mm];
	\draw(0,-1.5)[fill]circle[radius=1.0mm];
	\draw(-1.5,0)[fill]circle[radius=1.0mm];
	\node at(2.5,0){$K_{n-5}-M_{\frac{n-5}{2}}$};
	\node at(-1.5,0)[above=3pt]{$u$};
	\node at(0,1.5)[above=3pt]{$v_1$};
	\node at(0,-1.5)[below=3pt]{$v_2$};
	\node at(0,0.5)[left=2pt]{$v_4$};
	\node at(0,-0.5)[left=2pt]{$v_3$};
	\draw(-1.5,0)--(0,1.5);
	\draw(-1.5,0)--(0,-1.5);
	\draw(0,1.5)--(0,0.5);
	\draw(0,0.5)--(0,-0.5);
	\draw(0,-0.5)--(0,-1.5);
	\draw(0,1.5)--(1.175,0.5);
	\draw(0,0.5)--(1,0.15);
	\draw(0,-0.5)--(1,-0.15);
	\draw(0,-1.5)--(1.175,-0.5);
    \node at (1,-2.5) {$G_2^{(1)}$};
	\end{tikzpicture}
	\end{subfigure}
    \begin{subfigure}{.4\textwidth}
    \centering
    \begin{tikzpicture}
	\draw(4,0) ellipse [x radius=1.5cm, y radius=0.5cm];
	\draw(4,1.5) ellipse [x radius=1.2cm, y radius=0.5cm];
	\draw(4,-1.5) ellipse [x radius=1.2cm, y radius=0.5cm];
	\draw(1.5,1)[fill]circle[radius=1.0mm];
	\draw(1.5,-1)[fill]circle[radius=1.0mm];
	\draw(0,0)[fill]circle[radius=1.0mm];
	\node at(4,0){$K_{n-7}-M_{\frac{n-7}{2}}$};
	\node at(5.8,0){$T''$};
	\node at(5.4,2){$V_2$};
	\draw(3.5,1.5)[fill]circle[radius=1.0mm]--(4.5,1.5)[fill]circle[radius=1.0mm];
	\node at(3.5,1.5)[left=2pt]{$v_4$};
	\node at(4.5,1.5)[right=2pt]{$v_6$};
	\node at(5.4,-2){$V_1$};
	\node at(0,0)[below=2pt]{$u$};
	\node at(1.5,1)[left=2pt]{$v_1$};
	\node at(1.5,-1)[left=2pt]{$v_2$};
	\draw(3.5,-1.5)[fill]circle[radius=1.0mm]--(4.5,-1.5)[fill]circle[radius=1.0mm];
	\node at(3.5,-1.5)[left=2pt]{$v_3$};
	\node at(4.5,-1.5)[right=2pt]{$v_5$};
	\draw (0,0)--(1.5,1);
	\draw (0,0)--(1.5,-1);
	\draw (1.5,1)--(1.5,-1);
	\draw (4,1)--(4,0.5);
	\draw (4,-1)--(4,-0.5);
	\draw (1.5,1)--(2.8,1.5);
	\draw (1.5,-1)--(2.8,-1.5);
	\draw (1.5,1)--(3,0.37);
	\draw (1.5,-1)--(3,-0.37);
	\draw (5.2,-1.5) arc(-90:90:1.5);
    \node at (3.5,-2.5) {$G_2^{(2)}$};
	\end{tikzpicture}
    \end{subfigure}
\caption{$G_2^{(1)},G_2^{(2)}$}
	\label{fig:4}
\end{figure}
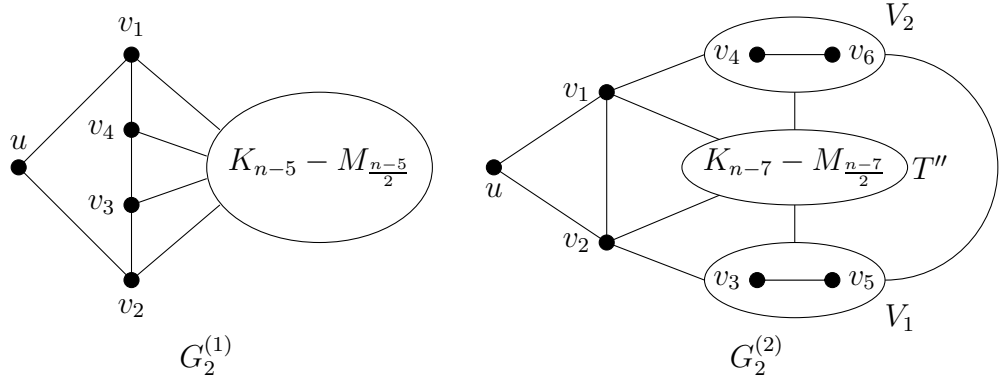

{\it Subcase 1.2.2.} $v_1v_2\in E(G)$. Suppose  $V(G)-N[v_1]=  \{v_3,v_5\}, V(G)-N[v_2]= \{v_4,v_6\}$.
Then $$N[v_3]= N[v_5]=V(G)\setminus \{u,v_1\},\quad N[v_4]=N[v_6]=V(G)\setminus \{u,v_2\}.$$
Let $T''=V(G)\setminus\{u,v_i, i=1,2,\cdots,6\}$.  Since $u\not\in N(v)$ and $ \{v_1,v_2,v_3,v_4,v_5,v_6\}\subseteq N(v)$ for all $v\in T''$, we  have $G[T'']=K_{n-7}-M_{(n-7)/2}$. Therefore,   $G\equiv G_2^{(2)}$ has the   diagram as shown in Figure \ref{fig:4}, which is isomorphic to $H_2(n)$.

Choose $v_5,v_6\in T'$ such that $v_5$ and $v_6$ are not adjacent in $G_2^{(1)}$. Then $G^{(2)}_2=G^{(1)}_2+v_1v_2+v_5v_6-v_1v_5-v_2v_6$. Let $X=(x_1,x_2,\ldots,x_n)^T$ be the Perron vector of $G^{(1)}_2$ with $x_i$ corresponding to the vertex $v_i$ for $i=1,\ldots,n$. Without loss of generality, we assume $v_i=i$ for $i=1,\ldots,6$. By symmetry, we have $x_{2k-1}=x_{2k}$ for $k=1,2,3$. Since
$$\rho(G^{(2)}_2)\ge X^TA(G^{(2)}_2)X,\quad \rho(G^{(1)}_2)=X^TA(G^{(1)}_2)X,$$
 we have $\rho(G^{(2)}_2)\ge\rho(G^{(1)}_2)+2(x_1-x_5)^2$.

Considering  the characteristic equation  of  $G^{(1)}_2$,
 we have
\begin{equation*}
\left\{
\begin{aligned}
  &\rho(G^{(1)}_2)x_1=\Sigma_{v_i\in T'}x_i+x_u+x_4=(n-5)x_5+x_u+x_3,\\
  &\rho(G^{(1)}_2)x_5=\Sigma_{v_i\in T'}x_i-x_5-x_6+x_1+x_2+x_3+x_4=(n-7)x_5+2x_1+2x_3,
\end{aligned}
\right.
\end{equation*}
which lead to $(\rho(G^{(1)}_2)+2)(x_1-x_5)=x_u-x_3$.
If $x_1=x_5$, then we get $x_u=x_3$, which contradicts Lemma \ref{lemm5}.
Therefore, we have $x_1\ne x_5$ and
 $$\rho(G^{(2)}_2)\ge\rho(G^{(1)}_2)+(x_1-x_5)^2>\rho(G^{(1)}_2).$$

Notice that $\pi=\{\{u\},\{v_1,v_2\},\{v_3,v_4,v_5,v_6\},T')\}$ is  an equitable partition of $V(G_2^{(2)})$ with quotient matrix
\begin{align*}
B_2&=
\begin{pmatrix}
0 & 2 & 0 & 0\\
1 & 1 & 2 & n-7 \\
0 & 1 & 3 & n-7\\
0 & 2 & 4 & n-9
\end{pmatrix},
\end{align*}
whose characteristic polynomial is
\begin{equation}\label{Eq02}
P(B_2,\lambda)=\lambda^4+(5-n)\lambda^3+(5-2n)\lambda^2+(3n-17)\lambda+2n-2\equiv f_2(\lambda).
\end{equation}
By the Perron-Frobenius theorem,  $\rho(G)=\rho(G_2^{(2)})=\rho(B_2)$ is the maximum real eigenvalue of $B_2$, as $B_2$ is irreducible.

{\it Subcase 1.3.} $3\le \delta\le n-5$.
Let $T_1=N(u),T_2=V(G)\setminus N[u]$. Then $\{T_1,T_2\}$ is a partition of $T$ with $|T_1|=\delta$ and $|T_2|=n-\delta-1$.
Suppose   $H$ is the complement of the graph $G-u$, which consists of
 $k$ connected components   $H_1, \ldots, H_k$.
Since $d_G(v)=\Delta=n-3$ for all   $v\in T$, we have
$$d_H(v)=\left\{\begin{array}{ll}
   2,& \text{ if } v\in T_1;\\
   1,&\text{ if } v\in T_2.
   \end{array}\right.$$
Therefore, for each $i\in \{1,\ldots,k\}$,  $H_i$ must be  one of the following three types.
	\begin{itemize}
\item\vskip -0.3cm{\bf Type (I)}: an edge with both ends from $T_2$;\par
	\item\vskip -0.3cm{\bf Type (II)}: a path of length $\ge 2$ with both ends from $T_2$ and all other vertices from $T_1$;\par
	\item\vskip -0.3cm{\bf Type (III)}: a cycle with all vertices from $T_1$.\par	
\end{itemize}
  Denoted by $m_1(G)$, $m_2(G)$ and $m_3(G)$ the numbers of  Type (I), Type (II) and Type (III) components in $\overline{G-u}$, respectively. We simply write $m_1,m_2,m_3$ when no confusion arises. By counting  the number of vertices from $T_1$ and $T_2$ in all components, we have
$$2m_2\leq 2(m_1+m_2)=|T_2|=n-\delta-1\quad \text{ and }\quad
	 m_2\le|T_1|=\delta,$$
which leads to
\begin{equation}\label{eqh1}
 m_2\le\frac{n-1}{3}.
 \end{equation}

Now we define an operation on the graph $G^{loop}$.
\begin{itemize}
{{\item   Operation 1.} Suppose   $H_{i}$ is a Type (II) component, say, $H_{i}=v_1v_2\cdots v_{t-1}v_t$ with $v_1,v_t\in T_2$ and $v_2,\ldots,v_{t-1}\in T_1$.
We construct a new graph  $\tilde{G}^{loop}$ by doing the following operation on $G^{loop}$:
\begin{itemize}
\item[(i)]
if $t>4$,  add the edges $v_1v_2,v_{t-1}v_t$ and delete the edges $v_1v_t,v_2v_{t-1}$;
\item[(ii)]
if $t=4$, add the edges $v_1v_2,v_2v_3,v_3v_4$,  delete the edge $v_1v_4$ and the loops incident to $v_2,v_3$;
\item[(iii)]if $t=3$, add the edges $v_1v_2,v_2v_3$,  delete the edge $v_1v_3$ and the loop  incident to $v_2$.
\end{itemize}}
\end{itemize}
Then
\begin{equation*}
\begin{aligned}
  &\Tilde{G}^{loop}=\left\{\begin{array}{ll}
  G^{loop}+v_1v_2+v_{t-1}v_t-v_1v_t-v_2v_{t-1},&\text{ if }t>4;\\
   G^{loop}+v_1v_2+v_2v_3+v_3v_4-v_1v_4-loop(v_2)-loop(v_3),&\text{ if } t=4;\\
 G^{loop}+v_1v_2+v_2v_3-v_1v_3-loop(v_2),&\text{ if }t=3.
 \end{array}\right.
\end{aligned}
\end{equation*}

Without loss of generality, we assume  $v_i=i$ for $i\in [t]$. Let $X=(x_1,x_2,\dots,x_n)^T$ be the Perron vector of $G$ and $ G^{loop} $ with $x_k$ corresponding to the vertex $k$ for $k=1,\ldots,n$, and
let $\tilde{X}=(\tilde{x}_1,\tilde{x}_2,\dots,\tilde{x}_n)^T$ be the Perron vector of $\Tilde{G}^{loop}$ with $\tilde{x}_k$ corresponding to the vertex $k$ for $k=1,\ldots,n$.
By symmetry, we have $$x_1=x_t,~~x_2=x_{t-1}\quad\text{ and }\quad \tilde{x}_1=\tilde{x}_t,~~\tilde{x}_2=\tilde{x}_{t-1}.$$
Since $$\rho(G^{loop})=X^TA(G^{loop})X\ge \tilde{X}^TA(G^{loop})\tilde{X}$$   \text{and}$$
 \rho(\Tilde{G}^{loop})=\tilde{X}^TA(\Tilde{G}^{loop})\tilde{X}\ge X^TA(\Tilde{G}^{loop})X,$$
we have
    \begin{equation*}
	  \rho(G^{loop})-\rho(\Tilde{G}^{loop}) \leq X^TA(G^{loop})X-X^TA(\Tilde{G}^{loop})X=2(x_1-x_2)^2,
	\end{equation*}
which leads to
	\begin{align}
	\rho(\Tilde{G}^{loop})\le \rho(G^{loop})\le\rho(\Tilde{G}^{loop})+2(x_1-x_2)^2.\label{1.1}
	\end{align}

Suppose   $\hat{x}=\max_{w\in V(G)} x_w$.  Applying Lemma \ref{lemma16}, we have
    \begin{align}
	&\hat{x}^2<\frac{n-3}{\rho^2(G)}.
    \label{1.3}
	\end{align}
 Since $x_1\in T_2$ and $x_2\in T_1$, by using the characteristic equation  of $G$,
we have
\begin{equation*}
\left\{
\begin{aligned}
  &\rho(G)x_1=\Sigma_{w\in V(G)}x_w-x_u-x_2-x_1,\\
  &\rho(G)x_2=\Sigma_{w\in V(G)}x_w-x_1-x_3-x_2,
  \label{1.2}
\end{aligned}
\right.
\end{equation*}
which lead to  $$\rho(G)(x_1-x_2)=x_3-x_u\le \hat{x}-x_u<\hat{x}.$$
Combining this with (\ref{1.1}), we have
	\begin{align}
	\rho(G^{loop})<\rho(\Tilde{G}^{loop})+\frac{2}{\rho^2(G)}\hat{x}^2.
    \label{1.5}
	\end{align}

Since $G$ is  nonregular, we have $\overline{d}(G)\le\rho(G)<\Delta(G)$.
It follows that
\begin{align}
  n-4<\frac{(n-1)(n-3)+\delta}{n}=\overline{d}(G)\leq\rho(G)<\Delta(G)=n-3.
  \label{1.6}
\end{align}
Combining  (\ref{1.3}), (\ref{1.5}) and (\ref{1.6}), we have
    \begin{align} \rho(G^{loop})<\rho(\Tilde{G}^{loop})+\frac{2(n-3)}{\rho^4(G)}<\rho(\Tilde{G}^{loop})+\frac{2(n-3)}{(n-4)^4}=\rho(\Tilde{G}^{loop})+\frac{2}{(n-4)^3}+\frac{2}{(n-4)^4}.\label{1.7}
    \end{align}
\indent
 Recall that $H$ has $m_2$ Type (II) components and $G$ has degree sequence $(\Delta,\ldots, \Delta,\delta)$. Denoted by $G^\delta$ the graph obtained by doing $m_2$ Operations 1 on $G^{loop}$ at the same time. Then by  \eqref{eqh1}, \eqref{1.1} and (\ref{1.7}), we have
    \begin{align}
    \rho(G^\delta) \leq \rho(G^{loop})<\rho(G^\delta)+\frac{2m_2}{(n-4)^3}+\frac{2m_2}{(n-4)^4}\le \rho(G^\delta)+\frac{2(n-1)}{3(n-4)^3}+\frac{2(n-1)}{3(n-4)^4}.\label{9}
	\end{align}
 Moreover, the complement of $G^\delta-u$ consists of only Type (I) and Type (III) components. So $\pi=(\{u\},T_1,T_2)$ is an equitable partition of $V(G^\delta)$.
Denoted by $B^{(\delta)} $ the quotient matrix of   $G^\delta$ corresponding to $\pi$ .
By Lemma \ref{lemm6} and \eqref{9},
we have $\rho(B^{(\delta)})=\rho(G^{\delta  })$ and
 \begin{align}\label{eq112}
	&\rho(B^{(\delta)}) \le \rho(G^{loop})<\rho(B^{(\delta)})+\frac{2(n-1)}{3(n-4)^3}+\frac{2(n-1)}{3(n-4)^4}.
	\end{align}

If $\delta\geq 3$, then
    \begin{align*}
    B_{\delta}& \equiv B^{(\delta)}-2I=
    \begin{pmatrix}
    0 & \delta & 0 \\
    1 & \delta-3 & n-\delta-1 \\
    0 & \delta & n-\delta-3
    \end{pmatrix}
    \end{align*}
has the characteristic polynomial
    \begin{equation*}\label{23}
	P(B_{\delta},\lambda)=\lambda^3+(6-n)\lambda^2+(9-3n)\lambda-\delta^2+(n-3)\delta.
	\end{equation*}

Since $3\leq \delta\leq n-5$,
 $B_{\delta}$ is an irreducible nonnegative matrix.
By the Perron-Frobenius theorem,  $\rho(B_\delta)$ equals the maximum real eigenvalue of $B_{\delta}$. Notice that
$$\rho(B^{(\delta)})=\rho(B_{\delta}+2I)=\rho(B_{\delta})+2\quad \text{ and }\quad \rho(G^{loop})=\rho(G)+2.$$
By \eqref{eq112} we have
\begin{align}
  &\rho(B_{\delta})\le \rho(G)<\rho(B_{\delta})+\frac{2(n-1)}{3(n-4)^3}+\frac{2(n-1)}{3(n-4)^4}\le \rho(B_{\delta})+\frac{1}{n^2},\label{10}
\end{align}
where the last inequality follows from $n\geq 59$.

 For any $3\le\delta_1<\delta_2=n-5$, we consider $P(B_{\delta_1},\lambda-\frac{1}{n^2})-P(B_{\delta_2},\lambda)$.
For any $\lambda\in [n-4,n-3+\frac{1}{n^2}]$, we have
    \begin{equation*}
	\begin{aligned}
P(B_{\delta_1},\lambda-\frac{1}{n^2})-P(B_{\delta_2},\lambda)
= & (\lambda-\frac{1}{n^2})^3+(6-n)(\lambda-\frac{1}{n^2})^2+(9-3n)
    (\lambda-\frac{1}{n^2})\\
    &-[\lambda^3+(6-n)\lambda^2+(9-3n)\lambda]
     +\delta_2^2-\delta_1^2+(n-3)\delta_1-(n-3)\delta_2.
   \end{aligned}
	\end{equation*}
Recall that $\delta_1$ and $\delta_2$ have the same parity. We have
    \begin{equation*}
	\begin{aligned}
&\delta_2^2-\delta_1^2+(n-3)\delta_1-(n-3)\delta_2  =(\delta_2-\delta_1)(\delta_1+\delta_2+3-n)\ge 2.
	\end{aligned}
	\end{equation*}
Hence,
    \begin{equation*}
	\begin{aligned}
P(B_{\delta_1},\lambda-\frac{1}{n^2})-P(B_{\delta_2},\lambda)\ge&-\frac{3\lambda^2}{n^2}+\frac{3\lambda}{n^4}-\frac{1}{n^6}-\frac{12\lambda}{n^2}+\frac{2\lambda}{n}
    +\frac{6}{n^4}-\frac{1}{n^3}+\frac{3}{n}-\frac{9}{n^2}+2\\
    =&\frac{1}{n^2}(-3\lambda^2+(2n-12+\frac{3}{n^2})\lambda)+2+\frac{3}{n}-\frac{9}{n^2}-\frac{1}{n^3}+\frac{6}{n^4}-\frac{1}{n^6}\\
    \geq& \frac{1}{n^2}(-3n^2+(2n-12+\frac{3}{n^2})n)+2+\frac{3}{n}-\frac{9}{n^2}-\frac{1}{n^3}+\frac{6}{n^4}-\frac{1}{n^6}\\
    =&1-\frac{9}{n}-\frac{9}{n^2}+\frac{2}{n^3}+\frac{6}{n^4}-\frac{1}{n^6}\\
    \geq& \frac{n^2-9n-9}{n^2}+\frac{2}{n^3}+\frac{6n^2-1}{n^6}\geq 0.
    \end{aligned}
    \end{equation*}

By Lemma \ref{lemma9} and (\ref{10}), we have
$$\rho(G_{\delta_1})<\rho(B_{\delta_1})+\frac{1}{n^2}\le\rho(B_{\delta_2})\leq \rho(G_{\delta_2})=\rho(G_{\delta_{n-5}}).$$
Therefore, if the graph $G\in \mathcal{G}(n,n-3)$ has  degree sequence $(\Delta,\dots,\Delta,\delta)$ with $3\le\delta\le n-5$, then $\delta= n-5$.

Suppose $u\in V(G)$ with $d(u)=n-5$. Then
 $T_1=N(u)$ and $T_2=V(G)\setminus N[u]$ have cardinalities $|T_1|=n-5$ and $|T_2|=4$.
Recalling the definition of $m_1,m_2,m_3$,
we have $m_1+m_2=2$.
Suppose $G\in \mathcal{G}(n,n-3)$  is a graph such that $m_1$   attains the minimum. We claim that $m_1=0$.

To the contrary, suppose $m_1>0$. If $m_3>0$, then
  $H$ contains a components $H_1$ of Type (I) and a component $H_2$ of  Type (III). Suppose $V(H_1)=\{v_1,v_2\}$ and $v_3v_4\not\in E(G)$ with $v_3,v_4\in V(H_2)$. Let $X=(x_1,x_2,\dots,x_n)^T$ be the Perron vector of $G$ with $x_i$ corresponding to the vertex $v_i$ for $i=1,2,3,4$.
By symmetry, we have $x_1=x_2, x_3=x_4$.
Let  $$G'=G+v_1v_2+v_3v_4-v_1v_3-v_2v_4.$$ Then by Lemma \ref{lemma7}, we have $\rho(G')\geq \rho(G)$ and hence $G'\in \mathcal{G}(n,n-3)$. Since  $m_1(G')=m_1-1$, we get a contradiction.
If $m_3=0$, then $(m_1,m_2,m_3)=(1,1,0)$ and the unique Type (II) component in $\overline{G-u}$ is $P_{n-3}$, say, $v_1v_2\cdots v_{t}$ with $t=n-3$. Let
$$G'=G +v_2v_3+v_{t-1}v_t-v_3v_{t-1}-v_2v_t.$$
Using the Perron vector of $G$ and the symmetry of $v_i$ and $v_{t-i+1}$ for $i=1,\ldots,t$,  we have $\rho(G')\ge \rho(G)$ and $G'\in \mathcal{G}(n,n-3)$. Moreover, we have $m_1(G')=m_3(G')=1$. By the case $m_3>0$, we can find another graph $G''\in \mathcal{G}(n,n-3)$  with $m_1(G'')=m_1-1$, a contradiction. Hence, we get $m_1=0$.

Now we get $(m_1,m_2)=(0,2)$. Suppose that $H$ contains a  Type (II) component $P_t=v_1v_2\cdots v_t$.
We  define a graph operation on $G^{loop}$ as follows.
\begin{itemize}
{\item Operation 2}:
\begin{equation*}
\Tilde{G}^{loop}=\left\{\begin{aligned}
  & G^{loop}+v_2v_3+v_3v_4-loop(v_3)-v_2v_4,&\text{if }t=4 ;\\
  & G^{loop}+v_2v_3+v_3v_4+v_4v_5-loop(v_3)-loop(v_4)-v_2v_5,&\text{if } t=5;\\
  &G^{loop}+v_2v_3+v_{t-1}v_t-v_3v_{t-1}-v_2v_t,&\text{if } t\geq 6.
\end{aligned}\right.
\end{equation*}
\end{itemize}

Next we prove
\begin{equation}\label{eqh3}
 \rho(\Tilde{G}^{loop})\geq \rho(G^{loop}).
 \end{equation}
 Let $X=(x_1,\ldots,x_n)$ be the Perron vector of   $G^{loop}$ with $x_i$ corresponding to $v_i$ for $i=1,\ldots,t$. By symmetry, we have $x_i=x_{t-i+1}$ for $i=1,\ldots,t$.
By Lemma \ref{lemma7}, we have
\begin{equation}\label{eqh4}
\begin{aligned}
  &\rho(G^{loop})=X^TA(G^{loop})X,\quad \rho(\Tilde{G}^{loop})\geq X^TA(\Tilde{G}^{loop})X
\end{aligned}
\end{equation}
When $t=4$ or $t\ge 6$, by \eqref{eqh4} we have  $$\rho(\Tilde{G}^{loop})-\rho(G^{loop})\ge X^T[A(\Tilde{G}^{loop})-A(G^{loop})]X=0,$$
which leads to \eqref{eqh3}. When $t=5$, let
$$\hat{G}=G^{loop}+v_2v_3+v_4v_5-v_3v_4-v_2v_5,$$
which is a multiple signed graph with a negative edge  $v_3v_4$.
Then $$\Tilde{G}^{loop}=\hat{G}+2v_3v_4-loop(v_3)-loop(v_4).$$
Let $\lambda(\hat{G})$ be an eigenvalue of $\hat{G}$ whose modular is $\rho(\hat{G})$, and let $\hat{X}=(\hat{x}_1,\ldots,\hat{x}_n)^T$ be an unit real eigenvector such that  $\lambda(\hat{G})=\hat{X}^TA(\hat{G})X.$  Notice that the subgraph $\hat{G}-v_3-v_4$ has minimum degree $n-5$. We have $$\rho(\hat{G})\ge \rho(\hat{G}-v_3-v_4)\ge n-5.$$
Considering the equation $\lambda(\hat{G}) \hat{X}=A(\hat{G})\hat{X}$, we have
\begin{equation*}
\left\{\begin{array}{l}
\lambda(\hat{G})\hat{x}_3=\sum_{i=1}^n\hat{x}_i+\hat{x}_3-3\hat{x}_4,\\
\lambda(\hat{G})\hat{x}_4=\sum_{i=1}^n\hat{x}_i+\hat{x}_4-3\hat{x}_3,
\end{array}
\right.
\end{equation*}
which implies $\hat{x}_3=\hat{x}_4$, as $\lambda(\hat{G})\ne 4$.

If $\lambda(\hat{G})>0$, then $\rho(\hat{G})=\lambda(\hat{G})\ge X^TA( \hat{G})X$. Hence, we have
$$\rho( \hat{G}) -\rho(G^{loop})\ge X^T[A( \hat{G})-A(G^{loop})]X=2(x_2x_3+x_4x_5-x_3x_4-x_2x_5)=0,$$
and
$$\rho(\Tilde{G}^{loop})-\rho( \hat{G}) \ge \hat{X}^T[A(\Tilde{G}^{loop})-A( \hat{G})]\hat{X}=2(\hat{x}_3\hat{x}_4 -\hat{x}_3^2-\hat{x}_4^2)=0,$$
which leads to \eqref{eqh3}. If $\lambda(\hat{G})<0$, then $\lambda(\hat{G})$ is the smallest eigenvalue of $G$. Denote by $ \lambda_{max}(\hat{G})$ the maximum eigenvalue of $G$, and by $\lambda_{min}$ the smallest eigenvalue of $\Tilde{G}^{loop}$. Then
$$ \lambda_{max}(\hat{G})\geq X^TA( \hat{G})X \quad \text{and}\quad  \lambda_{min}\le \hat{X}^TA(\Tilde{G}^{loop})\hat{X}.$$ Hence, we have
$$\lambda_{max}(\hat{G}) -\rho(G^{loop})\geq X^T[A( \hat{G})-A(G^{loop})]X=0$$
and
$$\lambda_{min}-\lambda( \hat{G}) \le \hat{X}^T[A(\Tilde{G}^{loop})-A( \hat{G})]\hat{X}=0,$$
which leads to
$$\rho(\Tilde{G}^{loop})\ge -\lambda_{min}\ge -\lambda( \hat{G})\ge \lambda_{max}(\hat{G})\geq \rho(G^{loop}).$$
Therefore, we get \eqref{eqh3}.
Notice that an Operation 2 on $G^{loop}$ generates a cycle and a $P_3$ in  $\overline{G^{loop}}$. By doing Operation 2 on $G^{loop}$  twice, we get a graph  $G_{n-5}^{loop}$ with with  vertex partition
$\pi=\{\{u\},T,\{w_1,w_2\},\{w_3,w_4,w_5,w_6\}\}$ such that $N(u)=T\cup\{w_1,w_2\}$, $\overline{G_{n-5}^{loop}-u}$ is the disjoint union of two paths $w_3w_1w_4$, $w_5w_2w_6$ and some cycles covering $T$. Then  by \eqref{eqh3}, we have $\rho(G_{n-5}^{loop})\geq \rho(G^{loop})$.

By direct computation,  $\pi$ is an equitable partition of $V(G_{n-5}^{loop})$ with quotient matrix $B_{n-5}'=B_{n-5}+2I$, where
  \begin{align}
B_{n-5}=&
\begin{pmatrix}
0 & n-7 & 2 & 0 \\
1 & n-10 & 2 & 4 \\
1 & n-7 & 1 &2\\
0 & n-7 & 1 & 3
\end{pmatrix} \nonumber
\end{align}
has characteristic polynomial
\begin{equation}\label{Eq05}
P(B_{n-5},\lambda)=\lambda^4+(6-n)\lambda^3+(8-3n)\lambda^2+(3n-18)\lambda+5n-17\equiv g(\lambda).
\end{equation}

Since $B_{n-5}$ is an irreducible nonnegative matrix,  $\rho(B_{n-5})=\rho(B'_{n-5})-2$ is the maximum real eigenvalue of $B_{n-5}$.
Recall that $\rho(G^{loop})=\rho(G)+2$. We have $$\rho(G)=\rho(G^{loop})-2\leq \rho(G_{n-5}^{loop})-2=\rho(B'_{n-5})-2=\rho(B_{n-5}).$$

Now we are ready to finish the proof for Case 1. It suffices to compare the maximum real root  of \eqref{Eq05} with those of \eqref{Eq01} and \eqref{Eq02}.

For $\lambda\ge n-3$, we have $f_2(\lambda)>0$  and $g(\lambda)>0$. Let $$t_1=n-3-\frac{2}{n}+\frac{4}{n^2}+\frac{5}{n^3}, \quad t_2=\frac{n}{2},\quad t_3=0,\quad t_4=-1-\frac{2}{n}-\frac{4}{n^2}. $$ If $n\geq 100$, then
\begin{eqnarray*}
f_2(t_1)&=&-3-\frac{35}{n}+\frac{244}{n^2}+\frac{52}{n^3}-\frac{969}{n^4}
-\frac{194}{n^5}+\frac{2076}{n^6}+\frac{718}{n^7}-\frac{2789}{n^8}\\
&&-\frac{1995}{n^9}+\frac{1400}{n^{10}}+\frac{2000}{n^{11}}+\frac{625}{n^{12}}
\leq -3-\frac{35}{n}+\sum_{i=2}^{i=12}\frac{3000}{n^i}<0,\\
g(t_1)&=&1-\frac{62}{n}+\frac{190}{n^2}+\frac{172}{n^3}-
\frac{817}{n^4}-\frac{420}{n^5}+\frac{1750}{n^6}+\frac{808}{n^7}
-\frac{2489}{n^8}\\
&&-\frac{1870}{n^9}+\frac{1400}{n^{10}}+\frac{2000}{n^{11}}+\frac{625}{n^{12}}
\geq 1-\frac{62}{n}-\sum_{i=2}^{i=12}\frac{3000}{n^i}>0,\\
 g(t_2)&=&-\frac{n^4}{16}+\frac{7n^2}{2}-4n-17=\frac{56-n^2}{16}n^2-4n-17<0,\\
g(t_3)&=&5n-17>0, \\
g(t_4)&=&-8 +\frac{16}{n} +\frac{72}{n^2}- \frac{32}{n^3}
-\frac{48}{n^4} + \frac{256}{n^6} + \frac{512}{n^7} + \frac{256}{n^8}
\leq -8+\frac{16}{n}+\sum_{i=2}^{i=8}\frac{600}{n^i}<0 .
\end{eqnarray*}
If $59\leq n\leq 99$, a direct computation (which we check by computers) also confirms that
\begin{equation}\label{eqh120}
g(t_4)<0, \quad g(t_3)>0,  \quad g(t_2)<0,  \quad g(t_1)>0.
\end{equation}
Since $t_4<t_3<t_2<t_1$, by \eqref{eqh120},  $g(\lambda)$ has at least three distinct roots in $(t_4,t_1)$. The fourth root of $g(\lambda)$ cannot be located in $(t_1,n-3)$, as both $g(t_1)$ and $g(n-3)$ are positive. On the other hand, since $f_2(t_1)<0$ and $f_2(n-3)>0$, $f_2(\lambda)$ has a root $\lambda_0\in (t_1,n-3)$.   Hence, the maximum real root  of $f_2(\lambda)$ is larger than the maximum real root   of $g(\lambda)$.
Moreover, since
$$f_1(\lambda)-f_2(\lambda)=(8-n)\lambda-n<0$$
for all positive $\lambda$, the maximum real root  of $f_1(\lambda)$ is larger than the maximum real root   of $f_2(\lambda)$.

Therefore, when $n$ is odd, $f_2(\lambda)$ attains the maximum real root among all the characteristic polynomials of nonregular connected graphs of order $n$ with maximum degree $n-3$, and $G\in \mathcal{G}(n,n-3)$  is isomorphic to $H_2(n)$.  If $n$ is even, $f_1(\lambda)$ attains the maximum real root among all the characteristic polynomials of nonregular connected graphs of order $n$ with maximum degree $n-3$, and $G\in \mathcal{G}(n,n-3)$  is isomorphic to $H_1(n)$.

{\it Case 2.} $|S|=2$.
Suppose that $V(G)=[n]$, $S=\{u,v\}$ and $G$ has  degree sequence
$(\Delta,\Delta,\ldots,\Delta,d_u,d_v)$, where $d_u=d_G(u),d_v=d_G(v)$ and $d_v\leq d_u\leq n-4$.
Let $X=(x_1,x_2,\ldots,x_n)^T$ be the Perron vector of $G$ with $x_k$ corresponding to the vertex $k$.
By Lemma \ref{lemm2}  and Lemma \ref{lemm4}, we have
\begin{equation*}\label{eqh5}
x_v \le x_u\quad \text{and}\quad  N_T(v)\subseteq N_T(u).
\end{equation*}
Let $$T_1=N(u)\cap N(v), \quad T_2=N(u)\setminus N[v],\quad T_3=V(G)\setminus N[\{u,v\}].$$
Then $\{T_1,T_2,T_3\}$ is a partition of $T$ with $|T_3|\geq 3$, as $N_T(v)\subseteq N_T(u)$ and $d_v\leq d_u\leq n-4$. Since $  \sum_{w\in V(G)}d(w)=(n-2)(n-3)+d_u+d_v$ is even,  both $d_u+d_v$ and $|T_2|=d_u-d_v$ are even.

For any $w\in T_3$, $d(w)=n-3$ implies that
$$T_1\cup T_2\subseteq N(w)\quad \text{for all}\quad w\in T_3.$$
Denoted by $H$ the complement of $G-u-v$, which consists of $k$ connected components $H_1,\ldots, H_k$.
Then similarly as in Case 1, $H_i$ is one of  Type (I), Type (II) and Type (III) for $i=1,\ldots,k$. (Here $T_1$ and $T_2$ are different from Case 1.)

We first prove that
\begin{equation}\label{eqh7}
d(v)=1\quad\text{or}\quad d(u)=d(v).
\end{equation}
Suppose \eqref{eqh7} does not hold. Then we can do the following   operations on $G$ for suitable vertices $t_1,t_2,t_3$ to get a connected nonregular graph $G'$ with maximum degree $n-3$.
\begin{itemize}
{\item Operation 3:} $G'=G+ut_1+vt_2-t_1t_2$, where $t_1,t_2$ are distinct vertices in $T_3$.

{\item Operation 4:} $G'=G-vt_1-ut_2+t_1t_2$, where $t_1\in T_1$ and $t_2\in T\setminus N[t_1]$.

{\item Operation 5:} $G'=G-vt_1+t_1t_2-t_2t_3+ut_3$, where $t_1\in T_1$, $t_2\in T\setminus N[t_1],t_3\in T_3$.
\end{itemize}

Since $G\in \mathcal{G}(n,n-3)$, we have $\rho(G')\le \rho(G)$ for all the above graphs $G'$.
  Let $X=(x_1,\ldots,x_n)$ be the Perron vector of $G$ with $x_i$ corresponding to the vertex $i$. Denote by  $$\lambda=\rho(G), \quad m=\min \{x_i|i\in T\},\quad M=\max \{x_i|i\in T\}.$$
Since $\bar{d}(G)\leq \lambda\leq \Delta$,  we have
$$n-5<\frac{(n-2)(n-3)+d_u+d_v}{n}=\bar{d}(G)\leq \lambda\leq \Delta=n-3.$$
Recall that $x_v\le x_u<x_t$ for all $t\in T$. Choose $i,j\in T$ such that $x_i=M$ and $x_j=m$.
Considering the characteristic equation  of $G$, we have
\begin{eqnarray*}
   \lambda x_i&\leq& \Sigma_{w\in V(G)}x_w-x_u-x_v-x_i,\label{eq24}\\
  \lambda x_j&\geq& \Sigma_{w\in V(G)}x_w-x_j-2M,\label{eq25}
\end{eqnarray*}
which lead to
\begin{equation*}
\begin{aligned}
  &(\lambda+1)M  \leq \Sigma_{w\in V(G)}x_w-x_u-x_v,  \\
  &(\lambda+1)m\geq \Sigma_{w\in V(G)}x_w-2M.
\end{aligned}
\end{equation*}
It follows that
\begin{equation}
\begin{aligned}
  &(\lambda+1)(M-m)\leq 2M-(x_u+x_v).\label{19}
\end{aligned}
\end{equation}

In Operation 4, $\rho(G')\le \rho(G)$ implies $-x_vx_{t_1}-x_ux_{t_2}+x_{t_1}x_{t_2}\leq 0$, which leads to
\begin{equation}\label{21}
 x_u+x_v\ge \frac{m^2}{M}.
\end{equation}
By (\ref{19}) and (\ref{21}), we have
\begin{equation*}
\begin{aligned}
  &(\lambda+1)(M-m)\leq 2M-\frac{m^2}{M},
\end{aligned}
\end{equation*}
which is equivalent to
\begin{equation*}
\begin{aligned}
  &t^2-(\lambda+1)t+(\lambda-1)\leq 0,~~\text{where}~~ t=\frac{m}{M}\leq 1.
\end{aligned}
\end{equation*}
It follows that
$$\frac{\lambda+1-\sqrt{\lambda^2-2\lambda+5}}{2}\leq t\leq 1,$$
which leads to
\begin{equation}
 \frac{M}{m}=\frac{1}{t}\leq \frac{\lambda+1+\sqrt{\lambda^2-2\lambda+5}}{2(\lambda-1)}< 1+\frac{1}{\lambda-1}+\frac{1}{(\lambda-1)^2},\label{24}
\end{equation}
as $\lambda^2-2\lambda+5<(\lambda-1+\frac{2}{\lambda-1})^2$.

In Operation 5, for convenience we assume $t_i=i$ for $i=1,2,3$. Since $\rho(G')\le \rho(G)$, we have
\begin{equation*}
   -x_vx_1+x_1x_2-x_2x_3+x_ux_3\leq 0,~~ i.e.,~~
   x_1(x_2-x_v)+x_3(x_u-x_2)\leq 0
  \end{equation*}
  Notice that $x_u-x_2<0, x_1\ge m$ and $ x_3\le M$. We get
 $$m(x_2-x_v)+M(x_u-x_2)\leq 0,~~ i.e.,~~
  Mx_u-mx_v+(m-M)x_2\leq 0,$$
 which forces
 \begin{eqnarray*}
 x_u-x_v\le x_u-\frac{m}{M}x_v\le \frac{1}{M}(Mx_u-mx_v)\leq \frac{x_2}{M}(M-m)\le M-m.
\end{eqnarray*}
Considering the characteristic equation of $G$ on the vertices $u$ and $v$,  we get
\begin{equation} \label{eq26}
  (d_u-d_v)m \leq (\lambda+1)(x_u-x_v)\leq (\lambda+1)(M-m).
\end{equation}
By (\ref{24}) and (\ref{eq26}) we have:
\begin{equation*}
\begin{aligned}
  &d_u-d_v \leq (\lambda+1)(\frac{M}{m}-1)< \frac{\lambda(\lambda+1)}{(\lambda-1)^2}\leq \frac{(n-3)(n-2)}{(n-6)^2}<2,
\end{aligned}
\end{equation*}
which contradicts the assumption that $d_u\ne d_v$, as $d_u-d_v$ is even.
 Therefore, we get \eqref{eqh7}. Next we distinguish two subcases.

{\it Subcase 2.1.}  $d(u)=d(v)$. Then  $ T_2=\emptyset$ and $\{T_1,T_3\}$ is a partition of $T$. Since $d(w)=n-3$ for all $w\in T$,
  $\overline{G[T_1]}$ is  2-regular, which forces   $\delta \equiv d(u)=d(v)\geq 4$.  We denote by $G_{\delta,\delta}$ the graph $G$ to indicate the degree $d(u)=d(v)=\delta$. Then
$\pi=(\{u,v\},T_1,T_3)$ is an equitable partition of $G_{\delta,\delta}$ with  quotient matrix
\begin{align*}
B_{\delta,\delta}&=
\begin{pmatrix}
1 & \delta-1 & 0 \\
2 & \delta-4 & n-\delta-1 \\
0 & \delta-1 & n-\delta-2 \\
\end{pmatrix},
 4\leq \delta \leq n-4,
\end{align*}
whose characteristic polynomial is
\begin{equation}\label{eq33}
P(B_{\delta,\delta},\lambda)=\lambda^3+(5-n)\lambda^2+(3-2n)\lambda-2\delta^2+(2n-4)\delta+n-3.
\end{equation}
By Lemma \ref{lemm6}, we have $\rho(G_{\delta,\delta})=\rho(B_{\delta,\delta})$, which is the maximum real root of \eqref{eq33}.

For any $\delta_1,\delta_2$ such that $4\le\delta_1<\delta_2\le n-4$, we have
\begin{align*}
P(B_{\delta_1,\delta_1},\lambda)-P(B_{\delta_2,\delta_2},\lambda)
=2(\delta_2-\delta_1)(\delta_1+\delta_2-n+2).\label{equ29}
\end{align*}
Applying Lemma \ref{lemm8}, we have
  $$\rho(B_{\delta_1,\delta_1})\le\rho(B_{\delta_2,\delta_2})\quad \text{for}\quad \delta_1+\delta_2\ge n-2$$
  and
  $$\rho(B_{\delta_1,\delta_1})\ge\rho(B_{\delta_2,\delta_2}) \quad \text{for}\quad \delta_1+\delta_2\le n-2$$
  with equality if and only if $\delta_1+\delta_2=n-2$. Thus, we have $$\rho(B_{n-4,n-4})>\rho(B_{\delta,\delta}) \quad \text{for all}\quad  4\le\delta<n-4.$$
  By \eqref{eq33}, we have
\begin{equation}\label{eq35}
P(B_{n-4,n-4},\lambda)=\lambda^3+(5-n)\lambda^2+(3-2n)\lambda+5n-19.
 \end{equation}

 When $n$ is odd, we compare the largest real root  of \eqref{Eq02} and \eqref{eq35}.
 For any $n-3\ge\lambda\geq \rho(B_2) \ge n-4+\frac{5}{n}=\bar{d}(G_2^{(2)})$, we have
\begin{align*}
\lambda P(B_{n-4,n-4},\lambda)-P(B_2,\lambda)=&-2\lambda^2+(2n-2)\lambda-2n+2\\
=&-2[(\lambda-\frac{n-1}{2})^2-\frac{n^2-6n+5}{4}]\\\ge&-2(n-3)^2+(2n-2)(n-3)-2n+2=2n-10>0,
\end{align*}
which implies $\rho(B_{n-4,n-4})<\rho(B_2)$.\par

When $n$ is even, we compare the largest real root  of \eqref{Eq01} and \eqref{eq35}. For any $n-3\ge\lambda\ge \rho(B_1)\geq n-4+\frac{4}{n}=\bar{d}(G_1)$, we have
\begin{align*}
\lambda P(B_{n-4,n-4},\lambda)-P(B_1,\lambda)=&-2\lambda^2+(3n-10)\lambda-n+2\\
=&-2[(\lambda-\frac{3n-10}{4})^2-(\frac{3n-10}{4})^2-\frac{n-2}{2}]\\
\ge&-2(n-3)^2+(3n-10)(n-3)-n+2\\=&3n^2-22n+38>0,
\end{align*}
which implies $\rho(B_{n-4,n-4})<\rho(B_1)$.

{\it Subcase 2.2.}  $d_v=1,d_u=\delta \geq 3$. We denote by $G_{\delta,1}$ the graph $G$. Then
$\pi=(\{v\},\{u\},N(u)\setminus \{v\},V(G)\setminus N[u])$ is an equitable partition of $G_{\delta,1}$ with the quotient matrix
\begin{align*}
B_{\delta,1}&=
\begin{pmatrix}
0 & 1 & 0 & 0 \\
1 & 0 & \delta-1 & 0 \\
0 & 1 & \delta-3 & n-\delta-1 \\
0 & 0 & \delta-1 & n-\delta-2 \\
\end{pmatrix},
 3\leq \delta\leq n-4,
\end{align*}
whose characteristic polynomial is
\begin{equation}\label{eq36}
P(B_{\delta,1},\lambda)=\lambda^4+(5-n)\lambda^3+(5-2n)\lambda^2+(-\delta^2+\delta n-\delta-3)\lambda-\delta+2n-5.
\end{equation}
By Lemma \ref{lemm6}, we have $\rho(G_{\delta,1})=\rho(B_{\delta,1})$, which is the maximum real root of \eqref{eq36}.

  For any  $\delta_1,
\delta_2$ such that  $ 3\le \delta_1<\delta_2 \le n-4$, we have
\begin{align*}
&P(B_{\delta_1,1},\lambda)-P(B_{\delta_2,1},\lambda)\\&=(-\delta_1^2+\delta_1 n-\delta_1+\delta_2^2-\delta_2 n+\delta_2)\lambda-\delta_1+\delta_2=(\delta_2-\delta_1)[(\delta_2+\delta_1-n+1)\lambda+1]
\end{align*}
which implies
$$\rho(B_{\delta_2,1})\ge \rho(B_{\delta_1,1})\quad \text{for}\quad \delta_1+\delta_2\ge n-1$$
and
 $$\rho(B_{\delta_2,1})\le \rho(B_{\delta_1,1})\quad \text{for}\quad \delta_1+\delta_2\le n-1,$$
 with equality if and only if $\delta_1+\delta_2=n-1$.

  If $n$ is odd, we have $\rho(B_{n-4,1})\ge\rho(B_{\delta,1})$ for $ 3\le \delta \le n-4$, as $\delta$ is odd. Since  $$P(B_{n-4,1},\lambda)-P(B_2,\lambda)=2\lambda-n+1>0\quad \text{for}\quad \lambda>\frac{n}{2},$$ applying Lemma \ref{lemma8}  we have $\rho(B_2)>\rho(B_{n-4,1})$.

  If $n$ is even, we have $\rho(B_{3,1})\ge\rho(B_{\delta,1})$ for  $ 3\le \delta \le n-5$. Since   $$P(B_{3,1},\lambda)-P(B_1,\lambda)=(n-6)(\lambda+1)>0,$$  applying Lemma \ref{lemma8}  we have $\rho(B_1)>\rho(B_{3,1})$.

  Therefore, we have $|S|\ne 2$. Combining this with  Case 1, we  complete  the proof. ~~~~  $ \square$
\end{proof}
	
\section*{Acknowledgement}
This work was supported by the National Natural Science Foundation of China (No.12171323), Guangdong Basic and Applied Basic Research Foundation (No. 2022A1515011995).

\end{document}